\documentclass[10pt]{amsart}
\usepackage{amsmath,amssymb,amsfonts,amsthm,amsopn}
\usepackage{latexsym,graphicx}
\usepackage{xcolor}
\usepackage{color, colortbl}
\usepackage{pb-diagram}
\usepackage[title]{appendix}
\usepackage{tikz-cd}
\usepackage{tikz}
\usepackage{mathtools}
\usepackage{enumerate}
\usepackage{hyperref}

\mathtoolsset{showonlyrefs}

\usepackage{multirow}

\newtheorem{theorem}{Theorem}[section]
\newtheorem{lemma}[theorem]{Lemma}
\newtheorem{corollary}[theorem]{Corollary}
\newtheorem{proposition}[theorem]{Proposition}
\newtheorem{definition}[theorem]{Definition}

\newtheorem{remark}[theorem]{Remark}
\theoremstyle{definition}
\newtheorem{example}[theorem]{Example}

\newcommand{\beqa}{\begin{eqnarray*}}
\newcommand{\eeqa}{\end{eqnarray*}}

\DeclareMathOperator*{\supp}{supp}
\DeclareMathOperator*{\sing}{sing}
\DeclareMathOperator*{\rank}{rank}
\DeclareMathOperator*{\id}{id}
\DeclareMathOperator*{\Sp}{Sp}

\DeclareMathOperator*{\GL}{GL}

\newcommand{\field}[1]{\mathbb{#1}}
\newcommand{\bR}{\field{R}}        
\newcommand{\bN}{\field{N}}        
\newcommand{\bC}{\field{C}}        
\def\la{\lambda}

\def\cF{\mathcal{F}}              
\def\cS{\mathcal{S}}
\def\cD{\mathcal{D}}

\def\rd{\bR^d}

\def\rdd{{\bR^{2d}}}

\def\<{\left<}
\def\>{\right>}

\def\mv1{M_v^1}

\def\mn{(m,n)}
\def\mn'{(m',n')}

\def\Ren{\mathbb{R}^d}

\def\f{\varphi}

\def\Sn2{S_{2}(L^{2}(\Ren))}
\def\S1{S_{1}(L^{2}(\Ren))}
\def\sig00{\sigma_{0,0}}

\def\la{\langle}
\def\ra{\rangle}

\begin{document}

\begin{abstract} 
Metaplectic operators form a relevant class of operators appearing in different applications, in the present work we study their Schwartz kernels. Namely, diagonality of a kernel is defined by imposing rapid off-diagonal decay conditions, and quasi-diagonality by imposing the same conditions on the smoothing of the kernel through convolution with the Gaussian. Kernels of metaplectic operators are not diagonal. Nevertheless, as we shall prove, they are quasi-diagonal under suitable conditions. Motivation for our study comes from problems in time-frequency analysis, that we discuss in the last section.

\end{abstract}

\title[Metaplectic operators with quasi-diagonal kernels]{Metaplectic operators with quasi-diagonal kernels}

\author{Gianluca Giacchi}
\address{Universit\'a della Svizzera Italiana, IDSIA, Facoltà di Informatica, Via la Santa 1, 6962 Lugano, Switzerland}
\email{gianluca.giacchi@usi.ch}
\author{Luigi Rodino}
\address{Universit\`a di Torino, Dipartimento di Matematica, via Carlo Alberto 10, 10123 Torino, Italy}
\email{luigi.rodino@unibo.it}

\thanks{}
\maketitle

\section{Introduction}
We aim to study the kernel of metaplectic operators, addressing in particular to their behavior on the diagonal. Let us first introduce the notions of diagonal and quasi-diagonal kernel. The so-called kernel theorem of L. Schwartz states that every linear continuous operator $T:\cS(\rd)\to\cS'(\rd)$ can be written in the integral form
\begin{equation}\label{LuigiIntro1}
	Tf(x)=\int_{\rd}k(x,y)f(y)dy, \qquad f\in\cS(\rd),
\end{equation}
the kernel $k$ being in general a distribution, with formal meaning of the integral. Different versions are present in the literature, in different functional frameworks for $f$, see for example \cite{treves2006topological}. Other standard setting for \eqref{LuigiIntro1} is $f\in\mathcal{D}(\Omega)$, with $\Omega$ open subset of $\rd$ and $T:\cD(\Omega)\to\cD'(\Omega)$.

For theoretical and applicative reasons, particular attention is given to the behavior of the kernels $k$ at the diagonal $\Delta=\{x=y\}$. The first striking result in this direction was the celebrated theorem of J. Peetre. A linear operator $T$ preserves the support of a function or distribution $f$, i.e. $\supp(Tf)\subseteq\supp(f)$, if and only if $T$ is a linear partial differential operator, the kernel $k$ satisfying
\begin{equation}\label{LuigiIntro2}
	\supp(k)\subseteq\Delta,
\end{equation}
see \cite{peetre1959une} for a precise functional setting. By replacing supports with singular supports, and assuming
\begin{equation}\label{LuigiIntro3}
	\sing\supp(k)\subseteq\Delta,
\end{equation}
under suitable hypotheses of regularity for $k$, one obtains the so-called pseudo-local property, i.e., $\sing\supp(Tf)\subseteq\sing\supp(f)$. Basic examples satisfying \eqref{LuigiIntro3} are the kernels of pseudo-differential operators, the literature on them being enormous, let us address to the monumental works of H\"ormander \cite{hormander1,hormander2,hormander3,hormander4}, where the notion of singular support is also refined by the notion of wave front set. When arguing in the global Euclidean setting, see for example \cite{ helffer1984book,shubin2001pseudodifferential}, the inclusion \eqref{LuigiIntro3} is replaced by estimates expressing an off-diagonal decay of the kernel $k$, let us formalize the definition in an elementary functional framework. Set as standard $\la x\ra=(1+|x|^2)^{1/2}$, $x\in\rd$, and define the vector space
\begin{equation}\label{LuigiIntro4}
	V(\rd)=\{f\in\mathcal{C}(\rd):\exists M\in\bN, C>0 \ \text{with} \ |f(x)|\leq C\la x\ra^{M} \}.
\end{equation}

\begin{definition}\label{defLuigiintro1}
	We say that $k\in\mathcal{C}(\rdd)$ is a \emph{diagonal kernel} if for every integer $N>0$ there exists $C_N>0$ such that
	\begin{equation}\label{LuigiIntro5}
		|k(x,y)|\leq C_N\la x-y\ra^{-N},\qquad x,y\in\rd.
	\end{equation}
\end{definition}

The corresponding operator $T$ in \eqref{LuigiIntro1} is then a linear map on $V(\rd)$. As counterpart of the notion of singular support we may define the asymptotic singular support $A-\sing\supp(f)$.

\begin{definition}\label{defLuigiintro2}
	Let $x_0\in\rd$, $x_0\neq0$, $f\in V(\rd)$. We say that $x_0\notin A-\sing\supp(f)$ if there exists a conic neighborhood $\Gamma\subseteq\rd$ of $x_0$ such that for every $N\in\bN$
	\begin{equation}\label{LuigiIntro6}
		|f(x)|\leq C_N\la x\ra^{-N}, \qquad x\in\Gamma,
	\end{equation}
	for suitable constants $C_N>0$.
\end{definition}

For $T$ with diagonal kernel as in Definition \ref{defLuigiintro1} we then have by standard arguments
\begin{equation}\label{LuigiIntro7}
A-\sing\supp(Tf)\subseteq A-\sing\supp(f),\qquad f\in V(\rd).
\end{equation}
Basic example of the preceding framework is given by time-frequency analysis of operators, see \cite{cordero2020time} and the bibliography therein. In fact, the Gabor matrix of a pseudo-differential operator is a diagonal kernel, and the A -- singular support coincides with the Gabor wave front set, see the last section of the paper.

It is natural trying to extend Definition \ref{defLuigiintro1} to a more general setting, namely to kernels $k\in\cS'(\rdd)$ corresponding to general continuous maps $T:\cS(\rd)\to\cS'(\rd)$. For, we perform a smoothing of $k$ by convolution with the Gaussian
\begin{equation}\label{LuigiIntro8}
	\f(t)=e^{-\pi|t|^2}, \qquad t\in\rdd
\end{equation}
providing the kernel in $\mathcal{C}^\infty(\rdd)$
\begin{equation}\label{regKernel}
	\tilde k(x,y)=k\ast \f(x,y), \qquad x,y\in\rd.
\end{equation}
We call $\tilde k$ the \emph{smoothed kernel} of $T$.

\begin{definition}\label{def-intro-2}
	We say that $k\in\cS'(\rdd)$ is \emph{quasi-diagonal} if the {smoothed kernel} $\tilde k$ is a diagonal kernel.
\end{definition}

Then \eqref{LuigiIntro7} is valid for $\tilde T$ with kernel $\tilde k$. Note that operators with diagonal kernel form an algebra, whereas composition of operators with quasi-diagonal kernel is not defined in general. Diagonality in Definition \ref{defLuigiintro1} implies quasi-diagonality, and the Peetre's condition $\mbox{supp}(k)\subseteq\Delta$ gives quasi-diagonality under suitable asymptotic conditions on $k$ at the diagonal.

Basic examples come again from time-frequency analysis. Namely, the Wigner kernel of a pseudo-differential operator is quasi-diagonal, cf. \cite{cordero2024wigner,cordero2025wigner,cordero2024WGM,cordero2022wigner} and the last section of this paper.

Independently of the applications to time-frequency analysis, it is interesting to discuss the quasi-diagonal property for kernels of different classes of operators. Here, we address \emph{metaplectic operators}, representing a very active field of research at this moment, cf.\cite{de2014metaplectic,dias2013metaplectic,fuhr2024metaplectic,giacchi2024boundedness,lerner2024uncertainty,mcnulty2024metaplectic,savin2022local,zhang2021uncertainty}, we refer to Section \ref{secPrel} below, and \cite{de2011symplectic,folland1989harmonic} for their definition. In short, the metaplectic group is a double cover of the symplectic group $\Sp(d,\bR)$ on $L^2(\rd)$. For every symplectic matrix $S\in\Sp(d,\bR)$ we can then compute a metaplectic operator $\hat S$ acting on $L^2(\rd)$, defined up to a phase factor. We have $\hat S:\cS(\rd)\to\cS(\rd)$, with extension from $\cS'(\rd)$ to $\cS'(\rd)$. Since the validity of the estimates \eqref{LuigiIntro5} does not depend on the phase factor, we shall omit in the following to specify its choice. In Section \ref{section3} we shall study the quasi-diagonality of the kernels of metaplectic operators. The results are somewhat technical, and we limit in this introduction to a statement in dimension $d=1$, i.e. the case of the so-called \emph{linear canonical transforms}. Consider then 
\begin{equation}
	\label{LuigiSq1}
		S=\begin{pmatrix}
			A & B\\
			C & D
		\end{pmatrix},
\end{equation}
with $A,B,C,D\in\bR$, $AD-BC=1$, and the corresponding metaplectic operator $\hat S$. 

\begin{proposition}\label{PropLuigiSq14}
	The kernel of $\hat S$ is quasi-diagonal if and only if $D=1$ or $C\neq0$. 
\end{proposition}

So, the only case of non-quasi-diagonality is when $C=0$ and $D\neq1$. To clarify the main issues in Section \ref{section3}, we test here Proposition \ref{PropLuigiSq14} on the following three basic examples:
\begin{equation}\label{LuigiSq2}
	S_1=\begin{pmatrix}
			D^{-1}& 0\\
			0 & D
		\end{pmatrix}, \quad 
		S_2=\begin{pmatrix}
			1 & B\\
			0 & 1
		\end{pmatrix}, \quad
		S_3=\begin{pmatrix}
			0 & 1\\
			-1 & 0
		\end{pmatrix},
\end{equation}
with $D\neq0$, $B\neq0$, corresponding to the metaplectic operators
\begin{equation}\label{LuigiSq3}
	\hat S_1f(x)=|D|^{1/2}f(Dx),\quad \hat S_2f(x)=|B|^{-1/2}e^{i\pi B^{-1}(\cdot)^2}\ast f(x), \quad \hat S_3f(x)=\cF f(x),
\end{equation}
where
\begin{equation}
	\cF f(x)=\int_{-\infty}^\infty f(y)e^{-2\pi i xy}dy, \qquad x\in\bR,
\end{equation}
is the Fourier transform of $f\in\cS(\bR)$. The kernel of the dilation map $\hat S_1$ is given in terms of the Dirac distribution by
\begin{equation}\label{LuigiSq4}
	k_1(x,y)=|D|^{1/2}\delta_{y-Dx}.
\end{equation}
By convolution with the Gaussian, we obtain
\begin{equation}\label{argumentGG}
	\tilde k_1(x,y)=\lambda_1 e^{-\lambda_2(y-Dx)^2},
\end{equation}
where $\lambda_1,\lambda_2$ are positive constants depending on $D$. It is then obvious that we have quasi-diagonality if and only if $D=1$. The analysis of the kernel of the convolution operator $\hat S_2$ is more interesting:
\begin{equation}\label{LuigiSq6}
	k_2(x,y)=|B|^{-1/2}e^{i\pi B^{-1}(y-x)^2},
\end{equation}
for which quasi-diagonality is not so evident. An easy computation gives
\begin{equation}\label{LuigiSq7}
	\tilde k_2(x,y)=\sigma_1e^{-\sigma_2(y-x)^2},
\end{equation}
where $\sigma_1,\sigma_2\in\bC$ depend on $B$, with $\Re(\sigma_2)>0$, so that we may again deduce diagonal estimates, according to Proposition \ref{PropLuigiSq14}. Finally, concerning the Fourier transform $\hat S_3$,
\begin{equation}\label{LuigiSq8}
	k_3(x,y)=e^{-2\pi ixy}.
\end{equation}
A direct computation, with the help of Lemma \ref{convGauss} below, gives
\begin{equation}
	\tilde k_3(x,y)=\frac{1}{\sqrt{2}}e^{-\pi(x^2+y^2)/2-i\pi xy}.
\end{equation}
So, $\tilde k_3\in\cS(\bR^2)$, and quasi-diagonality is obvious. 

When passing to the multi-dimensional case, $S$ in \eqref{LuigiSq1} is a $2d\times2d$ symplectic matrix. The conditions in Proposition \ref{PropLuigiSq14} generalize to $D=I$, $C\in\GL(d,\bR)$, and both are still sufficient for quasi-diagonality, see Propositions \ref{Cinv} and Corollary \ref{DI} in Section \ref{section3}. However, they are not any more necessary in general. Let us mention that kernels of operators are strictly and simply related to their Weyl symbols, studied extensively for metaplectic operators by M. de Gosson in \cite{de2005weyl}.

In Section \ref{section4} we propose some examples clarifying the connection with problems from time-frequency analysis.

\section{Preliminaries and notation}\label{secPrel}
We denote by $xy=x\cdot y$ the standard inner product of $\rd$, with $\cS(\rd)$ the Schwartz space of rapidly decreasing functions, and $\cS'(\rd)$ is the space of tempered distributions. We denote by $0$ the null matrix and by $I$ the identity matrix, regardless of their dimensions. If $M$ is a matrix, $\ker(M)$ and $R(M)$ denote its kernel and its range, respectively. If $M\in\bR^{d\times d}$ is a positive-semidefinite matrix, we denote by $\f_M(t)=e^{-\pi Mt\cdot t}$. If $M=I$, $\f:=\f_M$. We denote by $\GL(d,\bR)$ the group of $d\times d$ invertible matrices.

We denote the real part of $M\in\bC^{m\times n}$ by $\Re(M)$. Recall that if $M=M_1+iM_2$, with $M_1,M_2\in\bR^{d\times d}$, and $M_1$ is invertible, then 
\begin{equation}\label{invCmat}
	(M_1+iM_2)^{-1}=(M_1+M_2M_1^{-1}M_2)^{-1}-iM_1^{-1}M_2(M_1+M_2M_1^{-1}M_2)^{-1}.
\end{equation}
If $M\in\bR^{m\times n}$, $M^+$ denotes its Moore-Penrose inverse (pseudo-inverse). Recall that $M^+M=\id_{\ker(M)^\perp}$ and $MM^+=\id_{R(M)}$, and they are symmetric. Also, $(M^T)^+=(M^+)^T$. If $M\in\bR^{m\times n}$ has full-column rank, then $M^+=(M^TM)^{-1}M^T$. If $M:\bR^r\to\bR^d$ ($r\leq d$) maps an orthonormal basis of $\bR^r$ into an orthonormal basis of $R(M)$, then $M^+=M^T$.

\subsection{Gaussian integrals}

We make use of the following result \cite[Theorem A1]{folland1989harmonic}. 

\begin{theorem}\label{convGauss}
	Let $Q\in\bC^{d\times d}$ such that $Q^T=Q$ and $\Re(Q)$ is positive-definite. Then, for every $\xi\in\bC^d$,
	\begin{equation}\label{convGauss-formula}
		\int_{\rd}e^{-\pi Qt\cdot t}e^{-2\pi i\xi\cdot t}dt=\det(Q)^{-1/2}e^{-\pi Q^{-1}\xi\cdot\xi}.
	\end{equation}
	The branch of the square root is chosen so that $\det(Q)^{-1/2}>0$ when $Q$ is real and positive-definite.
\end{theorem}

\subsection{Symplectic matrices and metaplectic operators}
A matrix
\begin{equation}\label{blockS}
	S=\begin{pmatrix}
		A & B\\
		C & D
	\end{pmatrix}\in \bR^{2d\times 2d}
\end{equation}
is symplectic if the blocks $A,B,C,D\in\bR^{d\times d}$ satisfy
\begin{align}
	\label{GG1}
	&A^TC=C^TA,\\
	\label{GG2}
	&B^TD=D^TB,\\
	\label{GG3}
	&A^TD-C^TB=I,
\end{align}
or equivalently,
\begin{align}
	\label{GG4}
	&DC^T=CD^T,\\
	\label{GG5}
	&BA^T=AB^T,\\
	\label{GG6}
	&DA^T-CB^T=I.
\end{align}
The group of $2d\times2d$ symplectic matrices is denoted by $\Sp(d,\bR)$, and it is generated by the matrix of the standard symplectic form of $\bC^d$,
\begin{equation}\label{defJ}
	J=\begin{pmatrix}
		0 & I\\
		-I & 0
	\end{pmatrix}
\end{equation}
along with matrices in the form
\begin{equation}\label{defVPDE}
	V_P=\begin{pmatrix}
		I & 0\\
		P & I
	\end{pmatrix}, \quad \text{and} \quad 
	\cD_E=\begin{pmatrix}
		E^{-1} & 0\\
		0 & E^T
	\end{pmatrix},
\end{equation}
$P\in\bR^{d\times d}$, $P=P^T$ and $E\in\GL(d,\bR)$. In view of \eqref{GG1}, \eqref{GG2} and \eqref{GG3}, kernels and ranges of the blocks $A$, $B$, $C$ and $D$ in \eqref{blockS} are interrelated, cf. \cite[Corollary B.2, Table B1]{cordero2024hardy}. We also recall that $\Sp(1,\bR)=SO(2)$.

If $S\in\Sp(d,\bR)$ is \emph{free}, i.e.,  $B\in\GL(d,\bR)$ in \eqref{blockS}, then $B^{-1}A$ is symmetric by \eqref{GG5}. However, this is not the case if $B$ is singular, that is $B^+A$ is not symmetric in general, but it acts symmetrically on $\ker(B)^\perp$. 
\begin{lemma}\label{lemmaSymmetry}
	Let $S\in\Sp(d,\bR)$ have block decomposition \eqref{blockS}. Then,
	\[
		B^+Ax\cdot y=x\cdot B^+Ay, \qquad x,y\in\ker(B)^\perp.
	\]
	Consequently, if $W:\bR^r\to\ker(B)^\perp$ is a parametrization of $\ker(B)^\perp$, then $W^TB^+AW\in\bR^{r\times r}$ is symmetric.
\end{lemma}
\begin{proof}
	We prove that $B^+Ax\cdot y=x\cdot B^+Ay$ for every $x,y\in\ker(B)^\perp$. This implies that
	\[
		W^TB^+AWt\cdot s=t\cdot W^TB^+AWs, \qquad s,t\in\bR^r
	\]
	and, consequently, the assertion. Let $x,y\in\ker(B)^\perp$. Since $S$ is symplectic,
	\[
		AB^T=BA^T.
	\]
	Applying $B^+$ and $(B^T)^+$ to both sides, we get
	\[
		B^+AB^T(B^T)^+=B^+BA^T(B^T)^+.
	\]
	Since $B^T(B^T)^+$ is the identity on $R(B^T)=\ker(B)^\perp$, we have
	\begin{align*}
		B^+Ax\cdot y&=B^+AB^T(B^T)^+x\cdot y=B^+BA^T(B^T)^+x\cdot y\\
		&=A^T(B^T)^+x\cdot B^+By=A^T(B^T)^+x\cdot y,
	\end{align*}
	where we also used that that $BA^T$ and $B^+B$ are symmetric, and the latter is the identity on $\ker(B)^\perp$. Since $(B^T)^+=(B^+)^T$, the assertion follows.
	
\end{proof}

\begin{corollary}\label{corSymmetry}
	If $V_1:\bR^r\to R(B)$ is a parametrization of $R(B)$, then $V_1^TDB^+V_1$ is symmetric.
\end{corollary}
\begin{proof}
	It follows by applying Lemma \ref{lemmaSymmetry} to
	\begin{equation}
		S^{-1}=\begin{pmatrix}
		D^T & -B^T\\
		-C^T & A^T
		\end{pmatrix}.
	\end{equation}
	
\end{proof}

\subsection{Metaplectic operators} Consider the Schr\"odinger representation of the Heisenberg group: for $(x,\xi;\tau)\in\rd\times\rd\times\bR$ and $f\in L^2(\rd)$,
\begin{equation}
	\rho(x,\xi;\tau)f(t)=e^{2\pi i\tau}e^{-i\pi x\xi}e^{2\pi i\xi t}f(t-x).
\end{equation}
For every $S\in\Sp(d,\bR)$ 
there exists an operator $\hat S$ unitary on $L^2(\rd)$ such that
\begin{equation}\label{intertS}
	\hat S\rho(x,\xi;\tau)\hat S^{-1}=\rho(S(x,\xi);\tau).
\end{equation}
Any such operator is called \emph{metaplectic operator}. If $\hat S'$ is another unitary operator satisfying \eqref{intertS}, there exists a phase factor, i.e., $c\in\bC$ with $|c|=1$, so that $\hat S'=c\hat S$. Vice versa, any metaplectic operator $\hat S$ identifies uniquely the corresponding \emph{projection} $S$. The mapping $\hat S\in\{\hat S:S\in\Sp(d\bR)\}\mapsto S\in\Sp(d,\bR)$ is a homomorphism. 

\begin{example}\label{exampleMetap}
	\begin{enumerate}[(i)]
		\item The Fourier transform
		\begin{equation}
			\cF f(\xi)=\hat f(\xi)=\int_{\rd}f(x)e^{-2\pi i\xi x}dx, \qquad f\in\cS(\rd)
		\end{equation}
		is a metaplectic operator, and its projection is the matrix $J$ in \eqref{defJ}.
		\item For $E\in\GL(d,\bR)$, the operator
		\begin{equation}
			\mathfrak{T}_Ef(t)=|\det(E)|^{1/2}f(Et), \qquad f\in L^2(\rd)
		\end{equation}
		is metaplectic with projection $\cD_E$ defined as in \eqref{defVPDE}.
		\item For $P\in\bR^{d\times d}$, $P=P^T$, the chirp-product
		\begin{equation}
			\mathfrak{p}_Pf(t)=e^{i\pi Pt\cdot t}f(t), \qquad f\in L^2(\rd)
		\end{equation}
		is a metaplectic operator with projection $V_P$ as in \eqref{defVPDE}.
	\end{enumerate}
\end{example}

The operators in \eqref{exampleMetap} generate the group of metaplectic operators. We recall the following basic continuity property of metaplectic operators.

\begin{proposition}
	Metaplectic operators restrict to surjective isomorphisms of $\cS(\rd)$, and extend to surjective isomorphisms of $\cS'(\rd)$.
\end{proposition}

We will use the following integral representation of metaplectic operators, cf. \cite{cordero2024hardy,ter1999integral}.

\begin{lemma}
	Let $\hat S$ be a metaplectic operator with projection $S$ having blocks \eqref{blockS}, with $B\neq0$. For $x\in\rd$, write
	\begin{equation}
		x=x_1+x_2\in R(B)\oplus R(B)^\perp.
	\end{equation}
	Then, for every $f\in\cS(\rd)$,
\begin{equation}\label{TMO}
	{\hat Sf(x)=\mu_Se^{i\pi(DB^+x_1\cdot x_1+DC^Tx_2\cdot x_2)}\int_{\ker(B)^\perp}f(t+D^Tx_2)e^{i\pi B^+At\cdot t}e^{-2\pi i(B^+x_1-C^Tx_2)\cdot t}dt,}
\end{equation}
for a suitable constant $\mu_S$.
\end{lemma}
\begin{proof}
	It follows by \cite{ter1999integral} as explained in \cite[Corollary 3.2]{cordero2024hardy}. Indeed, even if in the latter the authors use $\xi_2\in A(\ker(B))$, the proof of Corollary 3.2 in \cite{cordero2024hardy} does not depend on the choice of the space complementing $R(B)$ in the direct sum decomposition of $\rd$. 
\end{proof}

\subsection{Properties of operators with diagonal kernels}\label{section24Luigi}
The following facts are well known in different contexts, for the benefit of the reader we collect and re-prove them with the notation of the Introduction. First, an operator $T$ with diagonal kernel maps $V(\rd)$ into $V(\rd)$. In fact, for a suitable $M$ and for arbitrary $N$, we have
\begin{equation}
	|Tf(x)|\leq \int_{\rd} |k(x,y)||f(y)|dy\leq C_N\int_{\rd} \la x-y\ra^{-N}\la y\ra^Mdy
\end{equation}
where here and in the sequel $C_N$ denotes different constants depending on $N$. By estimating 
\begin{equation}\label{Luigi241}
	\la x-y\ra^{-N}\leq C_N\la x\ra^{N}\la y\ra^{-N}, \qquad x,y\in\rd,
\end{equation}
we have
\begin{equation}
	|Tf(x)|\leq C_N\la x\ra^{N}\int_{\rd} \la y\ra^{-N+M}dy
\end{equation}
and, by choosing $N>M+d$, we have $Tf\in V(\rd)$.

Moreover, for the kernel $k$ of the product of operators $T_1$ and $T_2$ with diagonal kernels $k_1$ and $k_2$, we get the following estimates for $N>d$,
\begin{equation}
	|k(x,y)|\leq \int_{\rd} |k_1(x,z)||k_2(z,y)|dz\leq C_N\int_{\rd} \la x-z\ra^{-N}\la z-y\ra^{-N}dz\leq C_N\la x-y\ra^{-N}
\end{equation}
so that also $T_1T_2$ is an operator with diagonal kernel. Since convolution with Gaussians have diagonal kernels, this also shows that diagonality implies quasi-diagonality. 

Finally, let us prove the inclusion \eqref{LuigiIntro7} in the Introduction. Assume $x_0\notin A-\sing\supp(f)$, i.e., the estimates in Definition \ref{defLuigiintro2} are satisfied in the neighborhood $\Gamma$ of $x_0$. Take another conic neighborhood $\Gamma'$ with $\Gamma'\subset\subset \Gamma$ (we mean that the closure of $\Gamma'\cap\mathbb{S}^{d-1}$ is included in $\Gamma\cap\mathbb{S}^{d-1}$). For $x\in\Gamma'$ and $y\in\rd\setminus\Gamma$, we have
\begin{equation}
	\la x\ra+\la y\ra\leq C\la x-y\ra
\end{equation}
for some constant $C$, and then for every $N$
\begin{equation}\label{Luigi242}
	\la x-y\ra^{-2N}\leq C_N\la x\ra^{-N}\la y\ra^{-N}, \qquad x\in\Gamma',\quad y\in\rd\setminus\Gamma.
\end{equation}
We split 
\begin{equation}
	|Tf(x)|=I_1(x)+I_2(x),
\end{equation}
with
\begin{equation}
	I_1(x)=\left|\int_{\Gamma} k(x,y)f(y)dy\right|,\quad I_2(x)=\left|\;\int_{\rd\setminus\Gamma} k(x,y)f(y)dy\right|.
\end{equation}
We have, for every positive integers $P$ and $N$,
\[
	\la x\ra^PI_1(x)\leq C_{N,P}\int_{\Gamma}\la x\ra^P\la x-y\ra^{-P}\la y\ra^{-N}dy,
\]
by using the decay of $f$ in $\Gamma$ and the quasi-diagonality of $k$. According to \eqref{Luigi241}, we may estimate 
\begin{equation}
	\la x-y\ra^{-P}\leq C_P\la x\ra^{-P}\la y\ra^{P},
\end{equation}
and we conclude
\begin{equation}\label{Luigi243}
	\la x\ra^PI_1(x)\leq C_{N,P}\int_{\Gamma}\la y\ra^{-N+P}dy<\infty, \qquad x\in\rd,
\end{equation}
by choosing $N>P+d$. On the other hand, for $x\in\Gamma'$ and $y\in\rd\setminus\Gamma$, by using \eqref{Luigi242} to estimate the kernel, we have
\[
	\la x\ra^{P}I_2(x)\leq C_N\int_{\rd\setminus\Gamma}\la x\ra^{P-N}\la y\ra^{-N+M}dy<\infty,
\]
by choosing $N>\max\{P,M+d\}$. Hence, $x_0\notin A-\sing\supp(f)$, and \eqref{LuigiIntro7} is proved.

\subsection{Strategy of the proofs}
In the arguments of the subsequent sections, we will explicitly compute the kernel $k$ of the metaplectic operator $\hat S$ associated with the symplectic matrix $S$. The corresponding smoothed kernel $\tilde{k}$ will then be expressed in the form of a generalized Gaussian function. More precisely, we will show that it has the form
\begin{equation}\label{luigiRev21}
	|\tilde k(x,y)|=c\cdot e^{-\pi\mathcal{Q}_S(x,y)}, \qquad z\in\rdd,
\end{equation}
where $c>0$ and $\mathcal{Q}_S(x,y)$ is a positive semi-definite quadratic form. 
\begin{remark}\label{remPSD}
	Proving that $\mathcal{Q}_S$ is positive semi-definite is burdening. Fortunately, such a proof is not necessary. For, if $\mathcal{Q}_S$ had a negative eigenvalue, then $\mathcal{Q}_S(x,y)$ would exhibit exponential growth across the corresponding eigenspace. However, $\hat S$ maps $\mathcal{S}(\mathbb{R}^d)$ into itself, so that its kernel $k$ is a tempered distribution in $\mathcal{S}'(\mathbb{R}^{2d})$. As a result, the smoothed kernel $\tilde{k}$ defines a function of moderate growth, growing at most polynomially across every direction.
\end{remark}
Willing to discuss the asymptotic decay of $\tilde k$, we will be then led to the analysis of the manifold $\mathcal{Q}_S(x,y)=0$, let us fix the terminology as follows. 
\begin{definition}
	Given $S\in\Sp(d,\bR)$, considering $|\tilde k(x,y)|$ as in \eqref{luigiRev21}, we set
	\begin{equation}\label{luigiRev22}
		\Gamma_S=\{(x,y)\in\rdd, \; \mathcal{Q}(x,y)=0\}.
	\end{equation}
	We shall say that {\it  $\tilde k$ is localized on the manifold $\Gamma_S$}.
\end{definition}
With sloppy terms, we shall also say that $\tilde k$ {\it  decays off the manifold $\Gamma_S$}. In fact, when $\Gamma_S\subseteq\Delta$, then $\mathcal{Q}_S(x,y)>0$ for $x\neq y$, and we have
\begin{equation}\label{luigiRev23}
	\mathcal{Q}_S(x,y)\geq\varepsilon|x-y|^2, \qquad x,y\in\rd, \: x\neq y,
\end{equation}
for some $\varepsilon>0$. Hence, the quasi-diagonality estimate
\begin{equation}\label{luigiRev24}
	|\tilde k(x,y)|\leq C_N\la x-y\ra^{-N},
\end{equation}
follows with arbitrary $N$.

Note that our results in the sequel will identify also the cases where $\Gamma_S$ is not contained in the diagonal, see Theorem \ref{thmFinale} below. In this situation, the estimates \eqref{luigiRev24} should be replaced by more refined asymptotic bounds, allowing a generalization of the inclusion \eqref{LuigiIntro7}, and suggesting extensions of the definitions of diagonality and quasi-diagonality. This meets in a natural way the results of time-frequency analysis for Fourier integral operators, cf. \cite{cordero2024wigner,cordero2025wigner,cordero2020time}, and will be subject of future investigation.

\section{Smoothed kernels of metaplectic operators}\label{section3}

In this section, we study the quasi-diagonality of kernels of metaplectic operators. Preliminary investigations into their decay revealed an interplay between the blocks $C$ and $D$ in \eqref{blockS}. Specifically, in every test example we analyzed, either the invertibility of $C$ or the condition $D = I$ alone were sufficient for the smoothed kernel $\tilde k$ to exhibit quasi-diagonal decay. However, these conditions proved to be merely sufficient, the two blocks interacting in a more nuanced way in the general case. 

\begin{theorem}\label{thmFinale}
	Let $\hat S$ be a metaplectic operator with projection $S\in\Sp(d,\bR)$ having block decomposition \eqref{blockS}. Let $\tilde k$ be the smoothed kernel of $\hat S$, defined as in \eqref{regKernel}. Then, $\tilde k$ decays off the manifold
	\begin{equation}\label{defGamma}
		\Gamma_S=\{(x,D^Tx) : x\in R(C)^\perp\}.
	\end{equation}
\end{theorem}
We reserve Section \ref{section35GG} below for a brief discussion of $\Gamma_S$.\newline

Surprisingly, the block $B$, which appears in the integral representation of metaplectic operators, and demonstrated to be pivotal in various contexts such as Hardy's uncertainty principle or their boundedness on Lebesgue spaces, cf. \cite{cordero2024hardy,giacchi2024boundedness}, does not play any role in the off-diagonal decay of the smoothed kernel.


\subsection{Two sufficient conditions}\label{subsecSuffCond}
In our investigation, we began by examining metaplectic operators with $B = 0$ and $B \in \GL(d,\bR)$ in \eqref{blockS}. In both cases, if one of the following conditions was satisfied:
\begin{enumerate}[(i)] 
	\item $C$ is invertible, 
	\item $D = I$,
\end{enumerate}
then $k$ was quasi-diagonal. Here, we demonstrate that the kernel is quasi-diagonal under condition (i), regardless of the structure of $B$. We note that the following result is a straightforward corollary of Theorem \ref{thmFinale}. However, its proof does not require the detailed analysis used for Theorem \ref{thmFinale}. Specifically, we show that if $C\in GL(d,\bR)$, the smoothed kernel of $\hat S$ is in $\cS(\rdd)$. In this context, metaplectic operators whose associated symplectic projections have an invertible block $C$ play a role analogous to that of quadratic Fourier transforms in classical harmonic analysis, i.e., metaplectic operators with projections having the block $B$ invertible.

\begin{proposition}\label{Cinv}
If $C\in\GL(d,\bR)$, then the smoothed kernel $\tilde k$ is in $\cS(\rdd)$. In particular, $ k$ is quasi-diagonal.
\end{proposition}
\begin{proof}
If $C$ is invertible, then
\begin{equation}\label{SJDVV}
	S=J\cD_{C^{-1}}V_{-C^TA}JV_{-C^{-1}D}J,
\end{equation}
where $J$ and the other matrices in \eqref{SJDVV} are defined as in \eqref{defJ} and \eqref{defVPDE}.

Using that the projection $\hat S\mapsto S$ is a homomorphism, a straightforward computation shows that, up to a phase factor,
\begin{equation}
	\hat Sf(x)=\int_{\rd}f(y)k(x,y)dy,\qquad f\in\cS(\rd),
\end{equation}
with kernel
\begin{equation}
	k(x,y)=|\det(C)|^{-1/2}\cF\Phi_{-\tilde S}(x,y),
\end{equation}
where
\begin{equation}
	\tilde S=\begin{pmatrix}
		AC^{-1} & C^{-T}\\
		C^{-1} & C^{-1}D
	\end{pmatrix}.
\end{equation}
Then, up to a phase factor,
\begin{align}
	\tilde k(x,y)&=|\det(C)|^{-1/2}\cF\Phi_{-\tilde S}\ast\f(x,y)=|\det(C)|^{-1/2}\cF(\Phi_{-\tilde S}\f)(x,y).
\end{align}

Since $\Phi_{-\tilde S}\f\in\cS(\rdd)$, we also have $\tilde k\in\cS(\rdd)$. This proves the assertion. 

\end{proof}

The sufficiency of condition (ii) relies on the same technical calculations as those in Theorem \ref{thmFinale}. For the sake of discussion, we present this consequence here in advance.
\begin{corollary}\label{DI}
	If $D=I$, then $ k$ is quasi-diagonal. More precisely, $\tilde k$ is localized on the manifold
	\begin{equation}\label{kerCGamma}
		\Gamma_S =\{(x,x) : x\in\ker(C)\}.
	\end{equation}
\end{corollary}
\begin{proof}
	It is a straightforward consequence of Theorem \ref{thmFinale}.	
\end{proof}


\subsection{The kernel of metaplectic operators}

Recall that we can decompose $\rd$ in terms of the blocks of $S$ either as:
\begin{equation}\label{Rd1}
	\rd=\ker(B)^\perp\oplus D^T(R(B)^\perp),
\end{equation}
or as:
\begin{equation}\label{Rd2}
	\rd=R(B)\oplus R(B)^\perp,
\end{equation}
see \cite[Lemma 3.4]{cordero2024hardy} applied to $S^{-1}$.

Accordingly, for the rest of this section, for $x,y\in\rd$, we write
\begin{equation}\label{x1x2}
	x=x_1+x_2\in R(B)\oplus R(B)^\perp
\end{equation}
and
\begin{equation}\label{y1y2}
	y=y_1+y_2\in\ker(B)^\perp\oplus D^T(R(B)^\perp).
\end{equation}

\begin{theorem}\label{thmkern}
Let $\hat S$ be a metaplectic operator with projection $S$ having blocks \eqref{blockS}. 
\begin{enumerate}[(i)]
\item If $B\in\GL(d,\bR)$,
\begin{equation}\label{kS3}
	k(x,y)=|\det(B)|^{-1/2}e^{i\pi DB^{-1}x\cdot x}e^{i\pi B^{-1}Ay\cdot y}e^{-2\pi i B^{-1}x\cdot y},
	\end{equation}
\item If $B=0$, 
\begin{equation}
	k(x,y)=|\det(D)|^{1/2}e^{i\pi CD^Tx\cdot x}\delta_{y=D^Tx}(y).
\end{equation}
\item In all the other cases
\begin{equation}\label{kS1}\begin{split}
	k(x,y)&=\mu_Se^{i\pi(DB^+x_1\cdot x_1+DC^Tx_2\cdot x_2)}\int_{\ker(B)^\perp}\delta_{y=z+D^Tx_2}(y)e^{i\pi B^+Az\cdot z}e^{-2\pi iz\cdot(B^+x_1-C^Tx_2)}dz,
		\end{split}
\end{equation}
as an oscillatory integral.
\end{enumerate}
All the identities above hold up to a phase factor.
\end{theorem}

\begin{proof}
	It is well-known that if $B\in\GL(d,\bR)$, then
	\begin{align*}
		\hat Sf(x)=|\det(B)|^{-1/2}e^{i\pi DB^{-1}x\cdot x}\int_{\rd}f(y)e^{i\pi B^{-1}Ay\cdot y}e^{-2\pi i B^{-1}x\cdot y}dy, \qquad f\in\cS(\rd),
	\end{align*}
	and (i) follows trivially. Similarly, if $B=0$, we have
	\begin{align*}
		\hat Sf(x)=|\det(D)|^{1/2}e^{i\pi CD^Tx\cdot x}f(D^Tx), \qquad f\in\cS(\rd),
	\end{align*}
	i.e.,
	\begin{equation}\label{nucleoB0}
		\hat Sf(x)=|\det(D)|^{1/2}e^{i\pi CD^Tx\cdot x}\int_{\rd}f(y)\delta_{y=D^Tx}(y)dy,
	\end{equation}
	proving (ii) as well. To prove (iii) we use the integral expression of $\hat S$ \eqref{TMO}:
	\begin{align*}
		\hat Sf(x)&=\mu_Se^{i\pi(DB^+x_1\cdot x_1+DC^Tx_2\cdot x_2)}\int_{\ker(B)^\perp}f(y_1+D^Tx_2)e^{i\pi B^+Ay_1\cdot y_1}e^{-2\pi i(B^+x_1-C^Tx_2)\cdot y_1}dy_1\\
		&=\mu_Se^{i\pi(DB^+x_1\cdot x_1+DC^Tx_2\cdot x_2)}\\
		&\times\qquad \int_{\ker(B)^\perp}\int_{D^T(R(B)^\perp)}f(y_1+y_2)\delta_{y_2=D^Tx_2}(y_2)e^{i\pi B^+Ay_1\cdot y_1}e^{-2\pi i(B^+x_1-C^Tx_2)\cdot y_1}dy_1dy_2.
	\end{align*}
	By writing
	\begin{align}
		f(y)&=\int_{\rdd} f(z) e^{(y-z)\eta}dzd\eta,
	\end{align}
	we obtain
	\begin{align}
		\hat Sf(x)&=\int_{\rd}f(z)k(x,z)dz,
	\end{align}
	with $k$ given in \eqref{kS1}. This concludes the proof.
	
\end{proof}

\subsection{Proof of Theorem \ref{thmFinale}}\label{subsec:proof}
In view of Theorem \ref{thmkern}, to understand the general quasi-diagonal decay of the kernels of metaplectic operators we need to distinguish three cases: $B=0$, $B\in\GL(d,\bR)$ and $1\leq \rank(B)<d$ in \eqref{blockS}, respectively.
\subsubsection{Proof of the case $B\in\GL(d,\bR)$}
By \eqref{kS3} and Theorem \ref{convGauss}, for a suitable constant $c>0$,
	\begin{align}
		k\ast \f(x,y)&=c\cdot \f(x,y)\int_{\rd}e^{-\pi Q(u,v)\cdot(u,v)}e^{-2\pi i(u,v)\cdot(ix,iy)}dudv=c\cdot \det(M)^{-1/2}\f(x,y)e^{\pi Q^{-1}(x,y)\cdot(x,y)}\\
		&=c\cdot\det(M)^{-1/2}e^{-\pi (I-Q^{-1})(x,y)\cdot(x,y)},
	\end{align}
	where $Q$ is the complex symmetric matrix (with positive-definite real part)
	\begin{equation}
		Q=\begin{pmatrix}
			I-iDB^{-1} & iB^{-T}\\
			iB^{-1} & I-iB^{-1}A
		\end{pmatrix}=:I-iN.
	\end{equation}
	By \eqref{GG2}, $DB^{-1}$ is symmetric, and we see that $N$ can be factorized as
	\begin{equation}
			N=\underbrace{\begin{pmatrix}
				B^{-T} & 0\\
				0 & B^{-1}
			\end{pmatrix}}_{=:\Delta_B}\begin{pmatrix}
				D^T & -I\\
				-I & A
			\end{pmatrix}=\Delta_B\tilde S.
	\end{equation}
	By \eqref{invCmat},
	\begin{equation}
		Q^{-1}=(I+N^2)^{-1}+iN(I+N^2)^{-1},
	\end{equation}
	so that
	\begin{align}
		|\tilde k(x,y)|&=c \cdot e^{-\pi (I-(I+N^2)^{-1})(x,y)\cdot(x,y)}=e^{-\pi N^2(I+N^2)^{-1}(x,y)\cdot(x,y)}\\
		&=c\cdot e^{-\pi N(I+N^2)^{-1}N(x,y)\cdot(x,y)}=e^{-\pi (I+N^2)^{-1}\Delta_B\tilde S(x,y)\cdot \Delta_B\tilde S(x,y)}\\
		\label{eG1}
		&=c\cdot e^{-\pi \Delta_B^T(I+N^2)^{-1}\Delta_B\tilde S(x,y)\cdot\tilde S(x,y)},
	\end{align}
	where we used that $(I+N^2)^{-1}$ and $N$ are simultaneously diagonalizable, so they commute, and $\Delta_B\tilde S$ is trivially symmetric. Since $\Delta_B$ is invertible and $(I+N^2)^{-1}$ is positive-definite, the symmetric matrix $\Delta_B^T(I+N^2)^{-1}\Delta_B$ is also positive-definite. 

By \eqref{eG1}, $\tilde k$ decays off $\ker(\tilde S)$. It is straightforward to check that $(x,y)\in\ker(\tilde S)$ if and only if $x\in \ker(AD^T-I)$ and $y=D^Tx$. Since $AD^T-I=BC^T$ by \eqref{GG6}, and $B$ is invertible, 
	\begin{align}
		\ker(\tilde S)&=\{ (x,D^Tx)\in\rdd : x\in R(C)^\perp \}=\Gamma_S,
	\end{align}
	concluding the proof in this case.

\subsubsection{The case $B=0$}

	We use \eqref{nucleoB0}. For a suitable constant $c\in\bC$ with $|c|=1$,
	\begin{align}
		\tilde k(x,y)&=c\cdot|\det(D)|^{1/2}\int_{\rdd} \f(x-t,y-s)e^{i\pi CD^Ts\cdot s}\delta_{s=D^Tt}(s)dtds\\
		&=c\cdot|\det(D)|^{1/2}\f(x,y)\int_{\rdd} e^{-\pi|t|^2}e^{-\pi|s|^2}e^{2\pi tx}e^{2\pi sy}e^{i\pi CD^Tt\cdot t}\delta_{s=D^Tt}(s)dtds\\
		&=c\cdot|\det(D)|^{1/2}\f(x,y)\int_{\rd} e^{-\pi|t|^2}e^{-\pi|D^Tt|^2}e^{2\pi tx}e^{2\pi D^Tty}e^{i\pi CD^Tt\cdot t}dt\\
		&=c\cdot|\det(D)|^{1/2}\f(x,y)\int_{\rd} e^{-\pi(I+DD^T-iCD^T)t\cdot t}e^{-2\pi i(ix+iDy)t}dt\\
		&=c\cdot|\det(D)|^{1/2}\f(x,y)e^{\pi(I+DD^T-iCD^T)^{-1}(x+Dy)(x+Dy)}.
	\end{align}
	By \eqref{invCmat}, the real part of $(I+DD^T-iCD^T)^{-1}$ is
	\begin{align}
		Q&=(I+DD^T+CD^T(I+DD^T)^{-1}CD^T)^{-1}\\
		\label{middle}
		&=(I+DD^T+CD^T(I+DD^T)^{-1}DC^T)^{-1}.
	\end{align}
	Consequently,
	\begin{align}
		|\tilde k(x,y)|&=|\det(D)|^{1/2}\f(x,y)e^{\pi Q(x+Dy)(x+Dy) }=|\det(D)|^{1/2} e^{-\pi \tilde S(x,y)\cdot(x,y)},
	\end{align}
	where
	\begin{align}
		\tilde S&=\begin{pmatrix}
		I-Q & -QD\\
		-D^TQ & I-D^TQD
		\end{pmatrix},
	\end{align}
	which is positive semi-definite by Remark \ref{remPSD}. Clearly, $(x,y)\in\ker(\tilde S)$ if and only if
	\begin{align}
		\begin{cases}
			(I-Q)x-QDy=0,\\
			-D^TQx+(I-D^TQD)y=0.
		\end{cases}
	\end{align}
	Since both $Q$ and $D$ are invertible, the system above is equivalent to
	\begin{equation}\label{middle2}
		\begin{cases}
			y=D^{-1}Q^{-1}(I-Q)x,\\
			(D^{-1}Q^{-1}-D^{-1}-D^T)x=0.
		\end{cases}
	\end{equation}
	By using \eqref{middle} in the second equation in \eqref{middle2}, we have
	\begin{align}
	&(D^{-1}(I+DD^T+CD^T(I+DD^T)^{-1}DC^T)-D^{-1}-D^T)x=0\\
	&\Leftrightarrow(D^{-1}CD^T(I+DD^T)^{-1}DC^T)x=0\\
	&\Leftrightarrow(CD^T(I+DD^T)^{-1}DC^T)x=0\\
	&\Leftrightarrow C(I+D^{-1}D^{-T})^{-1}C^Tx=0.
	\end{align}
	The last equation is equivalent to $x\in\ker((I+D^{-1}D^{-T})^{-1/2}C^T)$, that is $x\in \ker(C^T)$. Hence, $(x,y)\in\ker(\tilde S)$ if and only if
	\begin{align}
		\begin{cases}
			Dy=(Q^{-1}-I)x,\\
			x\in\ker(C^T),
		\end{cases}
	\end{align}
	that is
	\begin{align}
		\begin{cases}
			Dy=(DD^T+CD^T(I+DD^T)^{-1}DC^T)x=DD^Tx,\\
			x\in R(C)^\perp.
		\end{cases}
	\end{align}
	In conclusion,
	\begin{equation}
		\ker(\tilde S)=\{(x,D^Tx):x\in R(C)^\perp\}=\Gamma_S,
	\end{equation}
	and we are done.
	

\begin{remark}
	The cases $B\in\GL(d,\bR)$ and $B=0$ together entail Proposition \ref{PropLuigiSq14} in the Introduction as a corollary, completely characterizing quasi-diagonal kernels of linear canonical transforms.
\end{remark}

\subsubsection{The general case}
We first prove the assertion under the simplifying assumption $D^T(R(B)^\perp)=\ker(B)$. By writing $u=u_1+u_2\in R(B)\oplus {R(B)^\perp}$ and $v=v_1+v_2\in \ker(B)^\perp\oplus {D^T(R(B)^\perp)}=\ker(B)^\perp\oplus \ker(B)$, it follows by \eqref{kS1} that
	\begin{align*}
		\tilde k(x,y)&=k\ast\f(x,y)=\int_{\rdd}k(u,v)\f(x-u)\f(y-v)dudv\\
		&=\mu_S\int_{\rd}\int_{\ker(B)^\perp}e^{i\pi B^+u_1\cdot u_1+DC^T u_2\cdot u_2}e^{i\pi B^+Az\cdot z}e^{-2\pi iz(B^+u_1-C^Tu_2)}\\
		&\qquad\qquad\times \f(x-u)\f(y_1-z)\f(y_2-D^Tu_2)dzdu,
	\end{align*}
where the simplifying assumption is used to factorize the Gaussians.	
	Since the dependence of integrand on $u_1$ and $u_1$ tensorizes, we can write
	\begin{equation}\label{eq22}\begin{split}
		\tilde k(x,y)&=\mu_S\iiint e^{i\pi(DB^+u_1\cdot u_1+DC^Tu_2\cdot u_2)}e^{i\pi B^+Az\cdot z}e^{-2\pi i(B^+u_1-C^Tu_2)\cdot z}\\
		&\qquad\qquad\times\f(x_1-u_1)\f(x_2-u_2)\f(y_1-z)\f(y_2-D^Tu_2)du_1du_2dz,
	\end{split}\end{equation}
	where the convergence of the integral is due to the Gaussians, and the integrals are taken over $R(B)$, $R(B)^\perp$ and $\ker(B)^\perp$. Let $V_1:\bR^r\to R(B)$, $V_2:\bR^{d-r}\to R(B)^\perp$ and $W:\bR^r\to\ker(B)^\perp$ be parametrizations so that $V_1^TV_1=I$, $V_2^TV_2=I$ and $W^TW=I$. 
	Then,
		\begin{equation}\label{eq23}\begin{split}
		\tilde k(x,y)&=\mu_S\f(x_1,x_2,y_1,y_2)\iiint e^{i\pi(V_1^TDB^+V_1r\cdot r+V_2^TDC^TV_2s\cdot s + W^TB^+AWt\cdot t)}\\
		&\qquad\qquad\times e^{-2\pi i(W^TB^+V_1r-W^TC^TV_2s)\cdot t} e^{-\pi V_1^TV_1r\cdot r}e^{-\pi V_2^TV_2s\cdot s}e^{-\pi W^TWt\cdot t}\\
		&\qquad\qquad\times e^{-\pi V_2^TDD^TV_2 s\cdot s}e^{2\pi V_1^Tx_1\cdot r}e^{2\pi V_2^T(x_2+Dy_2)\cdot s}e^{2\pi W^Ty_1\cdot t}drdsdt.
	\end{split}\end{equation}
	
	The integral in equation \eqref{eq23} can be written as 
	\begin{align*}
		\iiint e^{-\pi(G-iL)(r,t,s)\cdot(r,t,s)}e^{-2\pi i\zeta\cdot(r,t,s)}drdtds,
	\end{align*}
	where
	\begin{align}
		G-iL&=\left(\begin{array}{cc|c}
			I & 0 & 0 \\
			0 & I & 0 \\
			\hline
			0 & 0 & I+V_2^TDD^TV_2
		\end{array}\right)-i\left(\begin{array}{cc|c}
			V_1^TDB^+V_1 & {-}V_1^T(B^+)^TW & 0 \\
			{-}W^TB^+V_1 & W^TB^+AW &  W^TC^TV_2 \\
			\hline
			0 & V_2^TCW & V_2^TDC^TV_2
		\end{array}\right)
		\end{align}
		and
	\begin{equation}
		\zeta=i\begin{pmatrix}
			V_1^Tx_1 \\
			W^Ty_1\\
			V_2^T(x_2+Dy_2)
		\end{pmatrix}.
	\end{equation}
	By Lemma \ref{lemmaSymmetry} and Corollary \ref{corSymmetry}, $W^TB^+AW$ and $V_1^TDB^+V_1$ are symmetric. Moreover, we recall that $D^T:R(B)^\perp \to D^T(R(B)^\perp))$ is bijective, as a consequence of \cite[Lemma 3.4]{cordero2024hardy}. The mapping $D^TV_2$ is therefore a parametrization for $D^T(R(B)^\perp)$, and consequently $V_2^TDD^TV_2$ is positive-definite. Hence, we can apply Theorem \ref{convGauss}, yielding
	\begin{equation}
		\tilde k(x,y)=c\cdot\f(x,y)e^{\pi(G-iL)^{-1}(V_1^Tx_1,W^Ty_1,V_2^T(x_2+Dy_2))(V_1^Tx_1,W^Ty_1,V_2^T(x_2+Dy_2))},
	\end{equation}
	and consequently, by \eqref{invCmat},
	\begin{align}
		|\tilde k(x,y)|&=|c|\f(x,y)e^{\pi\Re((G-iL)^{-1})(V_1^Tx_1,W^Ty_1,V_2^T(x_2+Dy_2))(V_1^Tx_1,W^Ty_1,V_2^T(x_2+Dy_2))}\\
		\label{GoBackHere}
		&=|c|\f(x,y)e^{\pi(G+LG^{-1}L)^{-1}(V_1^Tx_1,W^Ty_1,V_2^T(x_2+Dy_2))(V_1^Tx_1,W^Ty_1,V_2^T(x_2+Dy_2))}\\
		&=|c|e^{-\pi \tilde S(V_1^Tx_1,W^Ty_1 , V_2^Tx_2, V_2^TDy_2)\cdot (V_1^Tx_1,W^Ty_1 , V_2^Tx_2, V_2^TDy_2)},
	\end{align}	
	with
	\begin{equation}
		\tilde S=
		\left(\begin{array}{cc|cc}
			I & 0 & 0 & 0\\
			0 & I & 0 & 0\\
			\hline
			0 & 0 & I & 0\\
			0 & 0 & 0 & E^{-1}
		\end{array}\right)
		-
			\left( \begin{array}{cc|c}
			I & 0 & 0 \\
			0 & I & 0 \\
			\hline
			0 & 0 & I \\
			0 & 0 & I 
		\end{array}\right)(G+LG^{-1}L)^{-1} \left(\begin{array}{cc|cc}
			I & 0 & 0 & 0\\
			0 & I & 0 & 0\\
			\hline
			0 & 0 & I & I
		\end{array}\right),
	\end{equation}
	where we set $E=V_2^TDD^TV_2$. Observe that here we used that $y_2=(D^TV_2)(D^TV_2)^+y_2$ and, since $D^TV_2$ has full column rank, $(D^TV_2)^+=(V_2^TDD^TV_2)^{-1}V_2^TD$. Again, the matrix $\tilde S$ is positive semi-definite by Remark \ref{remPSD}.
	
	We need to compute the kernel of $\tilde S$. Let us write
	\begin{equation}\label{blocksR-1}
		(G+LG^{-1}L)^{-1}=\begin{pmatrix}
			\Xi & b\\
			b^T & \delta
		\end{pmatrix},
	\end{equation}
	where $\Xi$, $b$ and $\delta$ are suitable matrices, that we leave implicit. We can compute 
	\begin{align}
		\left( \begin{array}{cc|c}
			I & 0 & 0 \\
			0 & I & 0 \\
			\hline
			0 & 0 & I \\
			0 & 0 & I 
		\end{array}\right)(G+LG^{-1}L)^{-1} \left(\begin{array}{cc|cc}
			I & 0 & 0 & 0\\
			0 & I & 0 & 0\\
			\hline
			0 & 0 & I & I
		\end{array}\right)&=
		\left(\begin{array}{c|cc}
			\Xi & b & b\\
			\hline
			b^T & \delta & \delta\\
			b^T & E^{-1}\delta & E^{-1}\delta
		\end{array}\right)=\left(\begin{array}{c|cc}
			\Xi & b & b\\
			\hline
			b^T & \delta & \delta\\
			b^T & \delta & \delta
		\end{array}\right)
	\end{align}
	and
	\begin{equation}
		\tilde S
		=
		\left(\begin{array}{c|cc}
			I-\Xi & -b & -b\\
			\hline
			-b^T & I-\delta & -\delta\\
			-b^T & -\delta & E^{-1}-\delta
		\end{array}\right),
	\end{equation}
	where we used that $\tilde S$ must be symmetric to simplify the matrices above.
	
	Let $X=(V_1^Tx_1,W^Ty_1)$, $Y=V_2^TDy_2$ and $Z=V_2^Tx_2$. The equations corresponding to $\ker(\tilde S)$ are
	\begin{equation}
		\begin{cases}
			(I-\Xi)X-b(Y+Z)=0,\\
			-b^TX+(I-\delta)Z-\delta Y=0,\\
			-b^TX - \delta Z+(E^{-1}-\delta) Y=0.
		\end{cases}
	\end{equation}
	By subtracting the second equation to the third, we find
	\begin{equation}\label{systemOfEqu}
		Z=E^{-1}Y,
	\end{equation}
	that is
	\begin{equation}
		V_2^Tx_2=(V_2^TDD^TV_2)^{-1}V_2^TDy_2=(D^TV_2)^+y_2
	\end{equation}
	or, equivalently,
	\begin{equation}\label{y2DTx2}
		y_2=D^Tx_2.
	\end{equation}
	Consequently, $\ker(\tilde S) \subseteq \mathfrak{M} = \{(x,y) : y_2 = D^T x_2\}$. Determining $\ker(\tilde S)$ explicitly requires the full expressions of the blocks in \eqref{blocksR-1}, which can lead to an unnecessarily intricate computation. Instead, we leverage \eqref{y2DTx2}, substituting it into \eqref{GoBackHere} to analyze the behavior of $|\tilde k|$ on $\mathfrak{M}$ and characterize the manifold where $|\tilde k|$ does not decay. By \eqref{GoBackHere}, with $y_2=D^Tx_2$, 
	\begin{equation}\label{togethere}
		|k(x,y)|=|c|e^{-\pi N(V_1^Tx_1,W^Ty_1,(I+V_2^TDD^TV_2)V_2^Tx_2)\cdot(V_1^Tx_1,W^Ty_1,(I+V_2^TDD^TV_2)V_2^Tx_2)},
	\end{equation}
	where we used that $V_2^T(I+DD^T)x_2=(I+V_2^TDD^TV_2)V_2^Tx_2$, and defined
	\begin{align}
		N&=G^{-1}
		-(G+LG^{-1}L)^{-1}=(G+LG^{-1}L)^{-1}\left((G+LG^{-1}L)G^{-1}
		-I\right)=(G+LG^{-1}L)^{-1}LG^{-1}L
		G^{-1}.
	\end{align}	
	The kernel of $N$ coincides with the kernel of $LG^{-1}LG^{-1}$. We need to solve
	\begin{equation}
		LG^{-1}LG^{-1}\begin{pmatrix}
			V_1^Tx_1\\
			W^Ty_1\\
			(I+V_2^TDD^TV_2)V_2^Tx_2
		\end{pmatrix}=0,
	\end{equation}
	that is
	\begin{equation}
		LG^{-1}L \begin{pmatrix}
			V_1^Tx_1\\
			W^Ty_1\\
			V_2^Tx_2
		\end{pmatrix}=0,
	\end{equation}
	or, equivalently, $(V_1^Tx_1,W^Ty_1,V_2^Tx_2)\in \ker(G^{-1/2}L)=\ker(L)$, i.e,
	\begin{equation}
		\begin{pmatrix}
			V_1^T(B^+)^T & 0 & 0\\
			0 & W^T & 0\\
			0 & 0 & V_2^T
		\end{pmatrix}
		\begin{pmatrix}
			D^T & -I & 0\\
			-B^+ & B^+A & C^T\\
			0 & C & DC^T
		\end{pmatrix}\begin{pmatrix}
			x_1\\
			y_1\\
			x_2
		\end{pmatrix}=0.
	\end{equation}
	Explicitly,
	\begin{equation}\label{systemGG1}
		\begin{cases}
			y_1-D^Tx_1\in \ker(B),\\
			-B^+x_1+B^+Ay_1+C^Tx_2\in\ker(B),\\
			Cy_1+DC^Tx_2\in R(B).
		\end{cases}
	\end{equation} 
	By \cite[Table B1]{cordero2024hardy}, we extrapolate $D^T(R(B))\subseteq \ker(B)^\perp$, and consequently $y_1-D^Tx_1\in\ker(B)^\perp$, which entails
	\begin{equation}\label{y1DTx1}
		y_1=D^Tx_1.
	\end{equation}
	By plugging \eqref{y1DTx1} in the remaining equations of \eqref{systemGG1}, and using \eqref{GG6}, we have
	\begin{equation}\label{systemGG2}
		\begin{cases}
			B^+BC^Tx_1+C^Tx_2\in\ker(B),\\
			CD^Tx_1+DC^Tx_2\in R(B).
		\end{cases}
	\end{equation}
	Since $BB^+$ is the identity on $R(B)$, the first equation of \eqref{systemGG2} is equivalent to
	\begin{equation}
		BC^Tx=0,
	\end{equation}
	that is $C^Tx\in\ker(B)$. On the other hand, the second equation of \eqref{systemGG2} is equivalent to $DC^Tx\in R(B)$. However, $D$ is an isomorphism between $\ker(B)$ and $R(B)^\perp$, forcing $DC^Tx=0$, i.e., $C^Tx\in\ker(D)$. By \eqref{GG3}, if $z\in\ker(B)\cap\ker(D)$, then $0=A^TDz-C^TBz=z$. Consequently, $BC^Tx=DC^Tx=0$ implies $C^Tx=0$.
	
	Moreover, recall that $D^T:R(B)\oplus R(B)^\perp\to \ker(B)^\perp\oplus D^T(R(B)^\perp)$ and it is bijective on each corresponding component of the two direct products. Therefore, $y_j=D^Tx_j$, for $j=1,2$, if and only if $y=D^Tx$. In synthesis, 
	
	\begin{equation}
		\ker(\tilde S)=\{(x,D^Tx):C^Tx=0\}=\Gamma_S,
	\end{equation} 
	concluding the proof of Theorem \ref{thmFinale} in the simplified case. In general, if $D^T(R(B)^\perp)\neq \ker(B)$, then we may write $S=\cD_PS'\cD_Q$, where $P,Q$ are orthogonal, the first $d-r$ columns of $Q$ form an orthononormal basis of $\ker(B)$ and the first $d-r$ rows of $P$ form an orthonormal basis of $R(B)^\perp$. Then, if $S$ has blocks as in \eqref{blockS},
	\begin{equation}
		S'=\begin{pmatrix}
			PAQ & PBQ\\
			PCQ & PDQ
		\end{pmatrix}
	\end{equation}
	satisfies the simplifying assumption. Let us delve into the details, let $\{u_1,\ldots,u_{d-r}\}$ be an orthonormal basis of $\ker(B)$ and $\{v_1,\ldots,v_{d-r}\}$ be an orthonormal basis of $R(B)^\perp$, where $d-r=d-rank(B)$. Let $Q$ be any orthogonal matrix so that its first $d-r$ columns are the vectors $u_j$'s ($j=1,\ldots,d-r$). Let $P$ be any orthogonal matrix so that the first $d-r$ rows of $P$ are the vectors $v_j$'s ($j=1,\ldots,d-r$). Then,
	\begin{equation}
		\ker(PBQ)=\mbox{span}\{e_1,\ldots,e_j\}_{j=1}^{d-r},
	\end{equation}
	being $e_j$ the $j$-th element of the canonical basis of $\rd$. Moreover,
	\begin{equation}
		R(PBQ)^\perp=\ker(Q^TB^TP^T)=\mbox{span}\{e_1,\ldots,e_j\}_{j=1}^{d-r}.
	\end{equation}
	Since $D^T$ is an isomorphism between the two spaces, and they coincide, we must have $D^T(R(B)^\perp)=\ker(B)$. From the metaplectic perspective,
	\begin{equation}
		\hat S=\mathfrak{T}_{P}\hat S' \mathfrak{T}_{Q}.
	\end{equation}
	The composition law of kernels gives that
	\begin{equation}
		k_{\mathfrak{T}_{P}\hat S'}(x,y)=\int_{\rd} k_{\mathfrak{T}_{P}}(x,z)k_{\hat S'}(z,y)dz=\int_{\rd}\delta_{z=P x}(z)k_{\hat S'}(z,y)dz=k_{\hat S'}(Px,y),
	\end{equation}
	and therefore, by using again the composition law, we find:
	\begin{equation}
		k_{S}(x,y)=k_{\mathfrak{T}_{P}\hat S'\mathfrak{T}_{Q}}(x,y)=k_{S'}(Px,Q^Ty).
	\end{equation}
	Then, 
	\begin{align}
		\tilde k_S\ast \f(x,y)&=\int_{\rdd}k_{S'}(Pu,Q^Tv) \f(x-u,y-v)dudv=\int_{\rdd}k_{S'}(u',v') \f(x-P^Tu',y-Qv')du'dv'\\
		&=\int_{\rdd}k_{S'}(u',v') \f(Px-u,Q^Ty-v)dudv=\tilde k_{S'}(Px,Q^Ty).
	\end{align}
	Since 
	\begin{equation}
		Px\in R(PCQ)^\perp=\ker(Q^TC^TP^T)
	\end{equation}
	if and only if
	\begin{equation}
		Q^TC^TP^T(Px)=0 \quad \text{if and only if} \quad C^Tx=0,
	\end{equation}
	we see that $\tilde k_{S'}(Px,Q^Ty)$ is localized on
	\begin{equation}
		\{Q^Ty=Q^TD^TP^T(Px) : Px\in R(PCQ)^\perp\}=\{y=D^T x : x\in R(C)^\perp\},
	\end{equation}
	and we are done.
	
\subsection{A counterexample in dimension $>1$} Proposition \ref{PropLuigiSq14} establishes that if $d=1$, $\Gamma_S\subseteq\Delta=\{x=y\}$ if and only if $C\in\bR\setminus\{0\}$ of $D=1$. The question arises whether this characterization extends to dimensions $d>1$, that is to say that the sufficient conditions of Section \ref{subsecSuffCond} are also necessary. The answer is negative, and the counterexample in dimension $d=2$ is straightforward: consider the symplectic interchange
\[
	\Pi_2=\begin{pmatrix}
		1 & 0 & 0 & 0\\
		0 & 0 & 0 & 1\\
		0 & 0 & 1 & 0\\
		0 & -1 & 0 & 0
	\end{pmatrix},
\]
that is the projection of the partial Fourier transform with respect to the frequency variable:
\begin{equation}
	\cF_2f(x,\xi)=\int_{-\infty}^\infty f(x,t)e^{-2\pi i\xi t}dt, \qquad f\in\cS(\bR^2).
\end{equation}
It is easy to see that $R(C)^\perp=\mbox{span}\{(1,0)\}$, with $D^Tx=x$ for every $x\in R(C)^\perp$, and $\Gamma_{\Pi_2}$ is strictly contained in $\Delta$.

\subsection{Comments on $\Gamma_S$ ($d>1$)}\label{section35GG}
Let us discuss briefly the manifold $\Gamma_S$ in dimension $d>1$. The example in the previous section illustrates that characterizing the symplectic matrices for which
\begin{equation}\label{inclusion}
	\Gamma_S\subseteq\Delta=\{x=y\}
\end{equation}
can be challenging if $d>1$, as $R(C)^\perp\cap \{D^Tx=x\}$ may be a proper subspace of both $R(C)^\perp$ and $\{D^Tx=x\}$. However, some observations can still be made. The matrix $D^T$ acts as an isomorphism between $R(C)^\perp$ and $\ker(C)$, cf. \cite{cordero2024hardy,ter1999integral}. In geometric terms, $\Gamma_S\subseteq R(C)^\perp\times \ker(C)$, and the manifold $\Gamma_S$ is the graph of $y=D^Tx$ as a function on $R(C)^\perp$. 

This fact, along with straightforward considerations, leads to the following proposition, which provides necessary conditions for the inclusion \eqref{inclusion} to hold. Recall that if $C\in\GL(d,\bR)$, then $\Gamma_S\subseteq\Delta$ by Proposition \ref{Cinv}, and there is nothing to say.

\begin{proposition}
	Assume $C\notin\GL(d,\bR)$. The following statements are equivalent.
	\begin{enumerate}[(i)]
		\item $\Gamma_S\subseteq\{x=y\}$.
		\item $D^T|_{R(C)^\perp}=\id_{R(C)^\perp}$.
		\item $\lambda=1$ is an eigenvalue of $D^T$, and $R(C)^\perp$ is contained in the corresponding eigenspace.
	\end{enumerate}
	Importantly, if (i) holds, then $\ker(C)=\ker(C^T)$.
\end{proposition}

\begin{remark}
	Even if the characterization provided by Proposition \ref{PropLuigiSq14} fails in dimension $>1$, we can still discuss geometrically the possible scenario in dimension $d=2$.
	\begin{enumerate}[(i)]
		\item If $C\in\GL(2,\bR)$, then $R(C)^\perp=\{0\}$. In this case, the kernel $k$ is quasi-diagonal, see Proposition \ref{Cinv} above.
		\item If $C=0$, then $R(C)^\perp=\bR^2$ and $\{y=D^Tx\}$ is a line contained in $R(C)^\perp$. In this case, quasi-diagonality is equivalent to $D=I$.
		\item If $\rank(C)=1$, then $R(C)^\perp=\ker(C^T)$ and $D^T(R(C)^\perp)=\ker(C)$ are lines.
		\begin{itemize}
			\item[(iii.1)] If $\ker(C)\cap \ker(C^T)=\{0\}$, then $\Gamma_S\cap\Delta=\{0\}$, which is strictly contained in $\Gamma_S$, and the kernel is not quasi-diagonal.
			\item[(iii.2)] If $\ker(C)=\ker(C^T)$, then $D^T|_{R(C)^\perp}$ is a dilation and quasi-diagonality is again equivalent to $D^T|_{R(C)^\perp}=\id|_{R(C)^\perp}$.
		\end{itemize}
	\end{enumerate}
\end{remark}

\section{Examples from time-frequency analysis}\label{section4}
We finally give some examples showing direct connections with problems from time-frequency analysis. To be definite, let us recall some definitions \cite{cordero2020time,grochenig2013foundations,hormander1,hormander2,hormander3}. For $f\in\cS'(\rd)$ and $g\in\cS(\rd)$, the \emph{short-time Fourier transform} is defined by
\begin{equation}\label{RodinoL1}
	V_gf(x,\xi)=\int_{\rd}f(t)\overline{g(t-x)}e^{-2\pi it\cdot \xi}dt, \qquad x,\xi\in\rd.
\end{equation}
Fix $g,\gamma\in\cS(\rd)\setminus\{0\}$. The \emph{Gabor matrix} of a linear continuous operator $T:\cS(\rd)\to\cS'(\rd)$ is defined to be 
\begin{equation}\label{RodinoL2}
	h(z,w)=\la T\pi(z)g,\pi(w)\gamma\ra, \qquad z=(x,\xi),w=(y,\eta)\in\rdd
\end{equation}
with $\pi(z)f(t)=e^{2\pi i t\cdot \xi}f(t-x)$, $t\in\rd$. The (cross-)Wigner distribution is defined by
\begin{equation}\label{RodinoL3}
	W(g,\gamma)(x,\xi)=\int_{\rd}g(x+t/2)\overline{\gamma(x-t/2)}e^{-2\pi it\cdot \xi}dt, \qquad x,\xi\in\rd.
\end{equation}
The meaning of \eqref{RodinoL3} extends to $g,\gamma\in\cS'(\rd)$ and for $f\in\cS'(\rd)$ one defines $Wf=W(f,f)$. The Wigner operator $K$ associated to $T$ is defined by 
\begin{equation}\label{RodinoL4}
	W(Tg,T\gamma)=K(W(g,\gamma))
\end{equation}
and its integral kernel $k$ is the Wigner kernel of $T$, \cite{cordero2022wigner,cordero2024understanding,cordero2024wigner}. We have $K:\cS(\rdd)\to\cS'(\rdd)$. The following identity links the Gabor matrix $h$ in \eqref{RodinoL2} and the Wigner kernel $k$, see \cite{cordero2024WGM}:
\begin{equation}\label{RodinoL5}
	|h|^2=k\ast (W\gamma\otimes Wg).
\end{equation}
Taking $g=\gamma=\f$, $\f(t)=e^{-\pi |t|^2}$, we have $W\f(x,\xi)=2^{d/2}e^{-2\pi(|x|^2+|\xi|^2)}$ and hence
\begin{equation}\label{RodinoL6}
	|h(z,w)|^2=|\la T\pi(z)\f ,\pi(w)\f\ra|^2=(k\ast\Phi)(z,w), \qquad z,w\in\rdd,
\end{equation}
with
\begin{equation}\label{RodinoL7}
	\Phi(z,w)=2^de^{-2\pi(|z|^2+|w|^2)}.
\end{equation}

So, independently of the different normalization of the Gaussian in \eqref{regKernel}, we have that the Wigner kernel is quasi-diagonal if and only if the Gabor matrix is diagonal. When $T$ is a pseudo-differential operator with symbol in the class $S^0_{0,0}(\rdd)$ of H\"ormander, its Wigner kernel is quasi-diagonal, hence from \eqref{RodinoL6}, we recapture the diagonality of the Gabor matrix. The A -- singular support in Definition \ref{defLuigiintro2} coincides with the Gabor wave front set. The following two examples involve metaplectic operators.

\begin{example}
	Let $S\in\Sp(d,\bR)$ be a given symplectic matrix and $\hat S$ the corresponding metaplectic operator, defined up to a phase factor. The Wigner kernel $k$ of the corresponding Wigner operator $K$ is given by
	\begin{equation}\label{RodinoL8}
		k(z,w)=\delta_{z-Sw},\qquad z,w\in\rdd,
	\end{equation}
	in view of the Wigner covariance, see for example \cite[Proposition 1.3.7]{cordero2020time}. We may regard $K$ as the metaplectic operator in dimension $2d$ associated to the symplectic matrix
	\begin{equation}\label{RodinoL9}
		\begin{pmatrix}
			S & 0\\
			0 & S^{-T}
		\end{pmatrix}.
	\end{equation}
By convolving the kernel \eqref{RodinoL8} with the Gaussian, we obtain the decay estimates
\begin{equation}\label{RodinoL10}
	|\tilde k(z,w)|=|(k\ast\Phi)(z,w)|\leq C_N\la w-Sz \ra^{-N},
\end{equation}
see \eqref{argumentGG} for the case $d=1$.
Similar estimates are satisfied by the Gabor matrix $h$, since $\tilde k=|h|^2$ in view of \eqref{RodinoL6}. We then recapture formula (6.95) in \cite{cordero2020time}, where $|h|$ was explicitly computed as
\begin{equation}\label{RodinoL11}
	|h|=|\psi(w-Sz)|
\end{equation}
with
\begin{equation}\label{RodinoL12}
	\psi=V_\f(\hat S\f)\in\cS(\rdd).
\end{equation}
\end{example}

\begin{example}
	The following Fourier multipliers appear in the study of the Wigner kernel of Fourier integral operators
	\begin{equation}\label{RodinoL13}
		\sigma(D)f(x)=\int_{\rd}e^{2\pi ix\cdot \xi}\sigma(\xi)\hat f(\xi)d\xi, \qquad x\in\rd,
	\end{equation}
	with $\sigma\in\mathcal{O}_M(\rd)$, the class of the slowly increasing functions of Schwartz \cite{schwartztheorie}, i.e. for every $\alpha\in\bN^d$,
	\begin{equation}\label{RodinoL14}
		|D^\alpha\sigma(\xi)|\leq C_\alpha\la\xi\ra^{N_{\alpha}}, \qquad \xi\in\rd,
	\end{equation}
	for suitable positive constants $C_\alpha$, $N_\alpha$, depending on $\alpha$. Note that the multiplication by $\sigma$ maps $\cS(\rd)$ to $\cS(\rd)$ and $\cS'(\rd)$ to $\cS'(\rd)$, hence $\sigma(D):\cS(\rd) \to \cS(\rd)$, $\cS'(\rd) \to \cS'(\rd)$. As basic examples of functions in $\mathcal{O}_M(\rd)$ we quote
	\begin{equation}\label{RodinoL15}
		\sigma(\xi)=e^{i\pi P_m(\xi)},
	\end{equation}	
	where $P_m(\xi)$ is a polynomial with real coefficients, of any order $m\in\bN$. For $m=2$, we retrieve in \eqref{RodinoL13} the metaplectic operators $\hat S$ corresponding to a symplectic matrix with blocks \eqref{blockS} having $A=D=I$ and $C=0$. One can write $\sigma(D)$ as a convolution operator with kernel
	\begin{equation}\label{RodinoL16}
		k(x,y)=\tau(x-y), \qquad \tau=\cF^{-1}(\sigma).
	\end{equation}
	If $\lambda\in\cS(\rd)$, in particular if $\lambda$ is a Gaussian, we have
	\begin{equation}\label{RodinoL17}
		\psi=\tau\ast\lambda=\cF^{-1}\cF(\tau\ast\lambda)=\cF^{-1}(\sigma\hat\lambda)\in\cS(\rd).
	\end{equation}
	According to the terminology of Schwartz \cite[Pages 244--245 and Theorem XV page 268]{schwartztheorie} we have that $\tau$ belongs to $\mathcal{O}'_c(\rd)$, the space of the rapidly decreasing distributions, which map by convolution $\cS(\rd)$ to $\cS(\rd)$ and $\cS'(\rd)$ to $\cS'(\rd)$. Note that $\tau$ is not in general a $\mathcal{C}^\infty$ function, even if $\sigma$ is of the form \eqref{RodinoL15} with $P(\xi)$ of order 2, cf. \cite[Lemma 4.4]{cordero2024unified}. The kernel $k(x,y)$ in \eqref{RodinoL16} is quasi-diagonal. In fact, by computing $\tilde k(x,y)=(k\ast \f)(x,y)$ in \eqref{regKernel}, we have
		\begin{equation}\label{RodinoL18}
		\tilde k(x,y)=\int_{\rd}\tau(\xi-\eta)e^{-\pi|x-\xi|^2}e^{-\pi|y-\eta|^2}d\xi d\eta.
	\end{equation}
	By the change of variables $\xi=\tilde\xi+\tilde\eta+y$, $\eta=\tilde\eta+y$, and application of the standard identity for convolution of Gaussians, we obtain
	\begin{equation}\label{RodinoL19}
	\begin{split}
		\tilde k(x,y)&=2^{-d/2}\int_{\rd}\tau(\tilde\xi)e^{-\pi|x-y-\tilde\xi-\tilde \eta|^2}e^{-\pi|\tilde\eta|^2}d\tilde\xi d\tilde\eta\\
		&=2^{-d/2}\int_{\rd}\tau(\tilde\xi)e^{-\pi|x-y-\tilde\xi|^2}d\tilde\xi=\psi(x-y)
	\end{split}
	\end{equation}
	where $\psi\in\cS(\rd)$ is given by \eqref{RodinoL17} with
	\begin{equation}\label{RodinoL20}
	\lambda(t)=2^{-d/2}e^{-\pi|t|^2/2}, \qquad t\in\rd.
	\end{equation}
	Hence, the estimates \eqref{LuigiIntro5} follow for $\tilde k$, and $k$ is quasi-diagonal. In the case where the order of $P_m(x)$ in \eqref{RodinoL15} is 2, we retrieve part of the result in Proposition \ref{Cinv}. For some explicit computations in the case $m=3$, we address \cite{cordero2025wigner}.
\end{example}

\section*{Acknowledgments}
The authors have been supported by the Gruppo Nazionale per l'Analisi Matematica, la Probabilità e le loro Applicazioni (GNAMPA) of the Istituto Nazionale di Alta Matematica (INdAM). Gianluca Giacchi has been supported by the SNSF starting grant ``Multiresolution methods for unstructured data” (TMSGI2 211684). We also thank the reviewers for their valuable comments and insightful suggestions, which have helped improve the clarity and rigor of the manuscript. 



%

%

\subsection*{Conflict of interest}

The authors declare no potential conflict of interests.

%
%
%



%


\begin{thebibliography}{10}

\bibitem{cordero2024hardy}
{\sc E.~Cordero, G.~Giacchi, and E.~Malinnikova}, {Hardy's uncertainty
  principle for {S}chr\"odinger equations with quadratic {H}amiltonians}, {\em J. Lond. Math. Soc.}, 111(4), e70134 (2025),  \url{https://doi.org/10.1112/jlms.70134}.

\bibitem{cordero2024understanding}
{\sc E.~Cordero, G.~Giacchi, and E.~Pucci}, {Understanding of linear
  operators through {W}igner analysis}, {\em J. Math. Anal. Appl.}, 543(1), 128955 (2025), \url{https://doi.org/10.1016/j.jmaa.2024.128955}.

\bibitem{cordero2024wigner}
{\sc E.~Cordero, G.~Giacchi, and L.~Rodino}, {Wigner analysis of operators.
  {P}art {I}{I}: {S}chr\"odinger equations}, {\em Comm. Math. Phys.}, 405 (2024),
  p.~156, \url{https://doi.org/https://doi.org/10.1007/s00220-024-04992-x}.

\bibitem{cordero2025wigner}
{\sc E.~Cordero, G.~Giacchi, and L.~Rodino}, {Wigner analysis of operators.
  {P}art {I}{I}{I}: Controlling ghost frequencies}, arXiv preprint
  arXiv:2412.01960, (2024).

\bibitem{cordero2024WGM}
{\sc E.~Cordero, G.~Giacchi, and L.~Rodino}, {Wigner kernel and {G}abor
  matrix of operators}, {\em Monatsh. Math.}, 1-15 (2025), \url{https://doi.org/10.1007/s00605-025-02070-5}.

\bibitem{cordero2024unified}
{\sc E.~Cordero, G.~Giacchi, and L.~Rodino}, {A unified approach to
  time-frequency representations and generalized spectrogram}, {\em J. Fourier Anal.
  Appl.}, 31 (2025),
  \url{https://doi.org/https://doi.org/10.48550/arXiv.2401.03882}.

\bibitem{cordero2020time}
{\sc E.~Cordero and L.~Rodino}, \emph{Time-{F}requency {A}nalysis of
  {O}perators}, vol.~75, Walter de Gruyter GmbH \& Co KG, 2020,
  \url{https://doi.org/https://doi.org/10.1515/9783110532456}.

\bibitem{cordero2022wigner}
{\sc E.~Cordero and L.~Rodino}, {Wigner analysis of operators. {P}art {I}:
  Pseudodifferential operators and wave fronts}, {\em Appl. Comput. Harmon. Anal.},
  58 (2022), pp.~85--123,
  \url{https://doi.org/https://doi.org/10.1016/j.acha.2022.01.003}.

\bibitem{de2005weyl}
{\sc M.~A. De~Gosson} {On the {W}eyl representation of metaplectic operators}, {\em Lett. Math. Phys.}, 72 (2005), pp.~129-142, \url{https://doi.org/10.1007/s11005-005-4391-y}.

\bibitem{de2011symplectic}
{\sc M.~A. De~Gosson}, \emph{Symplectic {M}ethods in {H}armonic {A}nalysis and
  in {M}athematical {P}hysics}, vol.~7, Springer Science \& Business Media,
  2011, \url{https://doi.org/https://doi.org/10.1007/978-3-7643-9992-4}.

\bibitem{de2014metaplectic}
{\sc M.~A. de~Gosson and F.~Luef}, {Metaplectic group, symplectic {C}ayley
  transform, and fractional {F}ourier transforms}, {\em J. Math. Anal. Appl.}, 416
  (2014), pp.~947--968,
  \url{https://doi.org/https://doi.org/10.1016/j.jmaa.2014.03.013}.

\bibitem{dias2013metaplectic}
{\sc N.~C. Dias, M.~A. De~Gosson, and J.~N. Prata}, {Metaplectic
  formulation of the {W}igner transform and applications}, {\em Rev. Math. Phys.}, 25
  (2013), p.~1343010,
  \url{https://doi.org/https://doi.org/10.1142/S0129055X13430101}.
  
  \bibitem{dopico2009parametrization}
{\sc F.~M. Dopico, and C.~R. Johnson}, {Parametrization of the matrix symplectic group and applications
}, {\em SIAM J. Matrix Anal. Appl.}, 31
  (2009), p.~650--673,
  \url{https://doi.org/10.1137/060678221}.

\bibitem{folland1989harmonic}
{\sc G.~B. Folland}, \emph{Harmonic {A}nalysis in {P}hase {S}pace}, Princeton
  {U}niversity {P}ress, 1989,
  \url{https://doi.org/https://doi.org/10.1515/9781400882427}.

\bibitem{fuhr2024metaplectic}
{\sc H.~F\"uhr and I.~Shafkulovska}, {The metaplectic action on modulation
  spaces}, {\em Appl. Comput. Harmon. Anal.}, 68 (2024), p.~101604,
  \url{https://doi.org/https://doi.org/10.1016/j.acha.2023.101604}.

\bibitem{giacchi2024boundedness}
{\sc G.~Giacchi}, {Boundedness of metaplectic operators within $ {L}^{p}$
  spaces, applications to pseudodifferential calculus, and time-frequency
  representations}, {\em J. Fourier Anal. Appl.}, 30 (2024), p.~69,
  \url{https://doi.org/https://doi.org/10.1016/j.acha.2023.101594}.

\bibitem{grochenig2013foundations}
{\sc K.~Gr\"ochenig}, \emph{Foundations of {T}ime-{F}requency {A}nalysis},
  Springer Science \& Business Media, 2013,
  \url{https://doi.org/https://doi.org/10.1007/978-1-4612-0003-1}.

\bibitem{helffer1984book}
{\sc B.~Helffer}, \emph{Th\'eorie {S}pectrale {P}our {D}es {O}p\'erateurs
  {G}lobalement {E}lliptiques}, no.~112 in Ast\'erisque, Soci\'et\'e
  math\'ematique de France, 1984,
  \url{http://www.numdam.org/item/AST_1984__112__R1_0/}.

\bibitem{hormander1}
{\sc L.~H\"ormander}, \emph{The {A}nalysis of {L}inear {P}artial {D}ifferential
  {O}perators {I}}, Classics in Mathematics, Springer, Berlin, Heidelberg,
  2003, \url{https://doi.org/https://doi.org/10.1007/978-3-642-61497-2}.

\bibitem{hormander2}
{\sc L.~H\"ormander}, \emph{The {A}nalysis of {L}inear {P}artial {D}ifferential
  {O}perators {I}{I}}, Classics in Mathematics, Springer, Berlin, Heidelberg,
  2005, \url{https://doi.org/https://doi.org/10.1007/b138375}.

\bibitem{hormander3}
{\sc L.~H\"ormander}, \emph{The {A}nalysis of {L}inear {P}artial {D}ifferential
  {O}perators {I}{I}{I}}, Classics in Mathematics, Springer, Berlin,
  Heidelberg, 2007,
  \url{https://doi.org/https://doi.org/10.1007/978-3-540-49938-1}.

\bibitem{hormander4}
{\sc L.~H\"ormander}, \emph{The {A}nalysis of {L}inear {P}artial {D}ifferential
  {O}perators {I}{V}}, Classics in Mathematics, Springer, Berlin, Heidelberg,
  2009, \url{https://doi.org/https://doi.org/10.1007/978-3-642-00136-9}.

\bibitem{lerner2024uncertainty}
{\sc N.~Lerner}, {On the uncertainty principle for metaplectic
  transformations}, {\em J. Funct. Anal.}, 289(5), 110997 (2025), \url{https://doi.org/10.1016/j.jfa.2025.110997}.

\bibitem{mcnulty2024metaplectic}
{\sc H.~McNulty}, {Metaplectic quantum time--frequency analysis, operator
  reconstruction and identification}, arXiv preprint arXiv:2411.01840,  (2024).

\bibitem{peetre1959une}
{\sc J.~Peetre}, {Une caract\'erisation abstraite des op\'erateurs
  diff\'erentiels}, {\em Math. Scand.}, 7 (1959), pp.~211--218,
  \url{http://www.jstor.org/stable/24489021} (accessed 2025-01-14).

\bibitem{savin2022local}
{\sc A.~Savin and E.~Schrohe}, {Local index formulae on noncommutative
  orbifolds and equivariant zeta functions for the affine metaplectic group},
  {\em Adv. Math.}, 409 (2022), p.~108624,
  \url{https://doi.org/https://doi.org/10.1016/j.aim.2022.108624}.

\bibitem{schwartztheorie}
{\sc L.~Schwartz}, \emph{Th{\'e}orie des {D}istributions}, Hermann, Paris, 1966.

\bibitem{shubin2001pseudodifferential}
{\sc M.~A. Shubin}, \emph{Pseudodifferential {O}perators and {S}pectral
  {T}heory}, Springer, 2~ed., 2001.

\bibitem{ter1999integral}
{\sc H.~ter Morsche and P.~J. Oonincx}, {On the integral representations
  for metaplectic operators}, {\em J. Fourier Anal. Appl.}, 8 (2002), pp.~245--257,
  \url{https://doi.org/https://doi.org/10.1007/s00041-002-0011-8}.

\bibitem{treves2006topological}
{\sc F.~Treves}, \emph{Topological Vector Spaces, Distributions and Kernels},
  Dover books on mathematics, Dover Publications, 2006.

\bibitem{zhang2021uncertainty}
{\sc Z.~Zhang}, \emph{Uncertainty principle of complex-valued functions in
  specific free metaplectic transformation domains}, {\em J. Fourier Anal. Appl.}, 27
  (2021), p.~68,
  \url{https://doi.org/https://doi.org/10.1007/s00041-021-09867-6}.

\end{thebibliography}
\end{document}